\documentclass[11pt]{article}

\usepackage{latexsym}
\usepackage{amssymb}
\usepackage{amsthm}
\usepackage{amscd}
\usepackage{amsmath}
\usepackage{mathrsfs}
\usepackage{graphicx}
\usepackage{shuffle}
\usepackage{hyperref}
\usepackage{mathtools}
\usepackage[bbgreekl]{mathbbol}
\usepackage{mathdots}
\usepackage{ytableau}
\usepackage[enableskew,vcentermath]{youngtab}
\usepackage{tikz-cd}

\usepackage[all]{xy}
\input xy \xyoption{frame}
\xyoption{dvips}

\usepackage[colorinlistoftodos]{todonotes}
\usetikzlibrary{chains,scopes}
\usetikzlibrary{shapes.geometric,positioning}

\DeclareSymbolFontAlphabet{\mathbb}{AMSb}
\DeclareSymbolFontAlphabet{\mathbbl}{bbold}

\numberwithin{equation}{section}
\allowdisplaybreaks[1]
\UseCrayolaColors

\theoremstyle{definition}
\newtheorem* {theorem*}{Theorem}
\newtheorem* {corollary*}{Corollary}
\newtheorem* {conjecture*}{Conjecture}
\newtheorem{theorem}{Theorem}[section]

\theoremstyle{definition}

\newtheorem* {example*}{Example}

\newtheorem{lemma}[theorem]{Lemma}
\theoremstyle{definition}
\newtheorem{definition}[theorem]{Definition}
\theoremstyle{definition}

\newtheorem{proposition}[theorem]{Proposition}

\newtheorem*{remark*}{Remark}

\theoremstyle{definition}
\newtheorem {example}[theorem]{Example}
\theoremstyle{definition}

\theoremstyle{definition}

\theoremstyle{definition}

\xyoption{dvips}

\def\({\left(}
\def\){\right)}

\def\ZZ{\mathbb{Z}}

\newcommand{\cN}{\mathcal{N}}

\def\fk{\mathfrak}

\def\barr{\begin{array}}
\def\earr{\end{array}}
\def\ba{\begin{aligned}}
\def\ea{\end{aligned}}
\def\be{\begin{equation}}
\def\ee{\end{equation}}

\def\qquand{\qquad\text{and}\qquad}
\def\quand{\quad\text{and}\quad}

\def\ds{\displaystyle}

\def\fkS{\fk S}

\def\ben{\begin{enumerate}}
\def\een{\end{enumerate}}

\def\fpf{{\textsf {FPF}}}

\def\Sfpf{\hat {\fk S}^\fpf}

\def\cHfpf{\mathcal{H}^{\textsf{Sp}}}
\def\iH{\mathcal{H}^{\textsf{O}}}

\def\iG{\mathfrak{G}^{\textsf{O}}}
\def\Gfpf{\mathfrak{G}^{\textsf{Sp}}}

\def\arcstart{\ \xy<0cm,-.15cm>\xymatrix@R=.1cm@C=.3cm }
\newcommand{\arcstartc}[1]{\ \xy<0cm,-.15cm>\xymatrix@R=.1cm@C=#1cm}

\def\ellhat{\widehat\ell}

\def\Sfpf{\widehat {\fk S}^\fpf}
\def\iS{\widehat \fkS}

\def\cG{\mathcal{G}}

\def\fkt{\fk t}

\newcommand{\bSym}{\overleftarrow{S}\hspace{-0.5mm}}
\newcommand{\bS}{\overleftarrow{\mathfrak{S}}\hspace{-0.5mm}}
\newcommand{\bSB}{\mathfrak{S}^{\textsf{B}}}%{\overleftarrow{\mathfrak{B}}\hspace{-0.5mm}}
\newcommand{\bSC}{\mathfrak{S}^{\textsf{C}}}%{\overleftarrow{\mathfrak{C}}\hspace{-0.5mm}}
\newcommand{\bSD}{\mathfrak{S}^{\textsf{D}}}%{\overleftarrow{\mathfrak{D}}\hspace{-0.5mm}}
\newcommand{\bSE}{(\mathfrak{S}^{\textsf{D}})^*}%\newcommand{\bSE}{\overleftarrow{\mathfrak{D}}^*}

\newcommand{\belem}{\overleftarrow{e}\hspace{-0.5mm}}

\newcommand{\bP}{\overleftarrow{P}\hspace{-0.5mm}}

\def\WBC{W^{\textsf{BC}}}
\def\WD{W^{\textsf{D}}}

\def\SigmaC{\Sigma_{\mathsf{BC}}}
\def\SigmaD{\Sigma_{\mathsf{D}}}

\def\comajc{\operatorname{comaj}_{\mathsf{BC}}}
\def\comaj{\operatorname{comaj}}
\def\exmaj{\operatorname{comaj}_{\mathsf{D}}}

\newcommand{\A}{\textsf{HeckeAtoms}}
\newcommand{\cR}{\textsf{Reduced}}
\newcommand{\Hecke}{\textsf{Hecke}}
\newcommand{\PrimedHecke}{\textsf{PrimedHecke}}
\newcommand{\InvolHecke}{\textsf{InvHecke}}
\def\CS{\mathsf{Compatible}}
\def\cN{\mathsf{NilCox}}
\def\cH{\mathsf{IdCox}_\beta}

\def\IdInvol{\mathsf{InvolMod}_\beta}
\def\IdFixed{\mathsf{FixedMod}_\beta}

\newcommand{\bI}{\mathcal{I}_n}
\newcommand{\bIfpf}{\mathcal{I}^{\mathsf{FPF}}_n}

\newcommand{\bFPF}{1^{\mathsf{FPF}}}

\def\fkG{\mathfrak{G}}

\newcommand{\biG}{\widehat{\fkG}}
\newcommand{\biGfpf}{\widehat{\fkG}^{\mathsf{FPF}}}
\newcommand{\bG}{\overleftarrow{\fkG}\hspace{-0.5mm}}

\newcommand{\bGB}{\fkG^{\textsf{B}}}%{\overleftarrow{\mathfrak{B}}\hspace{-0.5mm}}
\newcommand{\bGC}{\fkG^{\textsf{C}}}%{\overleftarrow{\mathfrak{C}}\hspace{-0.5mm}}
\newcommand{\bGD}{\fkG^{\textsf{D}}}%{\overleftarrow{\mathfrak{D}}\hspace{-0.5mm}}

\def\mFPF{m_1^{\mathsf{FPF}} }

\def\IP{\mathsf{InvDreams}}

\newcommand{\ltriang}{\raisebox{-0.5pt}{\tikz{\draw (0,0) -- (.25,0) -- (0,.25) -- (0,0);}}}

\newcommand{\ltriangneq}{\ltriang^{\!\!\neq}}

\def\bA{A^{(\beta)}}
\def\bB{B^{(\beta)}}
\def\bD{D^{(\beta)}}
\def\bC{C^{(\beta)}}
\def\bh{h^{(\beta)}}

\def\bfi{\mathbf{i}}

\numberwithin{equation}{section}
\allowdisplaybreaks[1]
\UseCrayolaColors

\makeatletter
\renewcommand{\@makefnmark}{\mbox{\textsuperscript{}}}
\makeatother

\usepackage{fullpage}

\begin{document}
\title{Principal specializations of Schubert polynomials in classical types}
\author{
Eric Marberg
\\ Department of Mathematics \\  HKUST \\ {\tt eric.marberg@gmail.com}
\and
Brendan Pawlowski \\ Department of Mathematics  \\ University of Southern California \\ {\tt br.pawlowski@gmail.com}
}
\date{}

\maketitle

\setcounter{tocdepth}{2}
%\tableofcontents

\begin{abstract}
There is a remarkable formula for the principal specialization of a type A Schubert polynomial as a weighted sum over reduced words. Taking appropriate limits transforms this to an identity for the backstable Schubert polynomials recently introduced by Lam, Lee, and Shimozono. This note identifies some analogues of the latter formula for principal specializations of Schubert polynomials in classical types B, C, and D. We also describe some more general identities for Grothendieck polynomials. As a related application, we derive a simple proof of a pipe dream formula for involution Grothendieck polynomials.
\end{abstract}

\section{Introduction}
\label{intro-sect}

There is a remarkable formula for the principal specialization $\fkS_w(1,q,q^2,\dots,q^{n-1})$ 
of a (type A) Schubert polynomial as a weighted sum over reduced words.
Originally a conjecture of Macdonald \cite{Macdonald}, this identity was first proved algebraically by Fomin and Stanley \cite{FominStanley}. Billey, Holroyd, and Young \cite{BHY,Young}
have recently found the first bijective proof of Macdonald's conjecture. 

In this note we identify some apparently new analogues of Macdonald's identity
for the principal specializations of Schubert polynomials in other classical types. 
Our methods are based on the algebraic techniques of Fomin and Stanley and 
 will also lead to a simple proof of (a $K$-theoretic generalization of) the main result of \cite{HMP6}.

To state our main theorems we need to recall a few definitions.
Throughout, we let $x_i$ for $i \in \ZZ$ be commuting indeterminates.
We use the term \emph{word} to mean a finite sequence $a_1a_2\cdots a_p$ whose letters belong to some totally ordered alphabet.
This alphabet will usually consist of 
the integers $\ZZ$ with their usual ordering, and in any case will always contain $(\ZZ,<)$ as a subposet.

\begin{definition} \label{compat-def}
A \emph{bounded compatible sequence} for a word $a=a_1a_2\cdots a_p$ 
is a weakly increasing sequence of integers $\bfi = (i_1 \leq i_2 \leq \dots \leq i_p)$
with the property that 
\[ i_j < i_{j+1}\text{ whenever } a_j \leq a_{j+1}
\qquand
 i_j \leq a_j \text{ whenever } 0 <i_j.\]
 Let $\CS(a)$ denote the set of all such sequences.
 Given $\bfi = (i_1  \leq \dots \leq i_p) \in \CS(a)$, define $x_\bfi = x_{i_1}\cdots x_{i_p}$
and write $0<\bfi$ if the numbers $i_1,\dots,i_p$ are all positive.
\end{definition}

Let $s_i = (i,i+1)$ denote the permutation of $\ZZ$ interchanging $i$ and $i+1$. % and fixing all other elements.
Fix a positive integer $n$ and let $S_n := \langle s_1,s_2,\dots,s_{n-1}\rangle \subset S_\ZZ := \langle s_i : i\in \ZZ\rangle$.
Both $S_n$ and $S_\ZZ$ are Coxeter groups with respect to their given generating sets.
A \emph{reduced word} for $w \in S_\ZZ$ is a word $a_1a_2\cdots a_p$ of shortest possible length such that 
$w = s_{a_1} s_{a_2}\cdots s_{a_p}$.
Let $\cR(w)$ denote the set of all such words.

\begin{definition}
The \emph{Schubert polynomial} of $w \in S_n$ is
\[\fkS_w := \sum_{a \in \cR(w)} \sum_{0<\bfi \in \CS(a)} x_\bfi \in \ZZ[x_1,x_2,\dots,x_{n-1}].\]
\end{definition}

Schubert polynomials are often defined inductively using divided difference operators, following the 
approach of Lascoux and Sch\"utzenberger.
The formula that we have given  is \cite[Thm. 1.1]{BJS}.
The identity of Macdonald \cite{Macdonald} mentioned at the start of this introduction is as follows.

\begin{theorem}[{Fomin and Stanley \cite[Thm. 2.4]{FominStanley}}]
\label{fs-thm}
If $w \in S_n$ then
\[
\fkS_w(1,q,q^2,\dots,q^{n-1}) = \sum_{a=a_1a_2\cdots a_p \in \cR(w)} \tfrac{[a_1]_q[a_2]_q\cdots [a_p]_q}{[p]_q!} q^{\comaj(a)}.
\]
where $\comaj(a) := \sum_{a_i < a_{i+1}} i$ and $[a]_q := \frac{1-q^a}{1-q}$ and $[p]_q! := [p]_q\cdots [2]_q[1]_q$.
\end{theorem}

Taking appropriate limits transforms the preceding formula into an identity for 
the \emph{backstable Schubert polynomials},
which may be defined as follows.
%Let $\bSym_n := \langle s_i : i <n\rangle$ be the Coxeter group of permutations $w \in S_\ZZ$ with $w(i)=i$ for all $i>n$.

\begin{definition}
The \emph{backstable Schubert polynomial} of $w \in S_n$ is
\[ \bS_w := 
\sum_{a \in \cR(w)} \sum_{\bfi \in \CS(a)} x_{\bfi} \in \ZZ[[\dots,x_{-1},x_0,x_1,\dots,x_{n-1}]].\]

\end{definition}

This is the same as the formula for $\fkS_w$ except now $\bfi=(i_1\leq i_2 \dots\leq i_p)$ may contain non-positive integers.
If $w \in S_n$ then $\bS_w(\dots,0,0,x_1,x_2,\dots,x_{n-1}) = \fkS_w$, 
while $\bS_w(\dots,x_{-2}, x_{-1},x_0,0,0,\dots,0)$ is the \emph{Stanley symmetric function} of $w$ in the variables $x_i$ for $i\leq 0$ \cite[Thm. 3.2]{LamLeeShim}. 

Note that $\bS_w$ is usually not a polynomial.
These power series were introduced by Lam, Lee, and Shimozono \cite{LamLeeShim} in connection with Schubert calculus on infinite flag varieties.  
They also arise as cohomology classes of degeneracy loci in products of flag varieties \cite{Pawlowski2019}.

If $F \in \ZZ[[\dots,x_{-1},x_0,x_1,\dots,x_{n-1}]]$ is homogeneous then 
 the formal power series $F(x_i \mapsto q^{i-1})$ obtained by setting $x_i = q^{i-1}$ for all integers $i<n$ is well-defined.
 The following result is easy to derive from Theorem~\ref{fs-thm} and is also a special case of Theorem~\ref{groth-thm}.
In this statement, for a word $a=a_1a_2\cdots a_p$ we write $\sum a := \sum_{i=1}^p a_i$ and $\ell(a) := p$.
\begin{theorem}\label{intro-thm-A}
If $w \in S_n$ then
$
\bS_w(x_i \mapsto q^{i-1}) = \sum_{a \in \cR(w)} \tfrac{q^{\sum a+\comaj(a)}}{(q-1)(q^{2}-1)\cdots (q^{\ell(a)}-1)}
$
where the right hand expression is interpreted as a Laurent series in $q^{-1}$.
\end{theorem}

%\begin{proof}
%Fix $w \in \bSym_n$ of length $p$, so that $\fkS_w$ is homogeneous of degree $p$.
%Let $1^m \times w \in \bSym_{m+n}$ be the permutation mapping $i \mapsto w(i-m) + m$ for all $i \in \ZZ$.
% Then 
% $a_1a_2 \cdots a_{p} \in \cR(w)$ if and only if $(a_1+m)(a_2+m)\cdots (a_p+m) \in \cR(1^m \times w)$, so we have 
%    \begin{align*}
%        \bS_w(x_i \mapsto q^{i-1}) &= \lim_{m \to \infty} \fkS_{1^m \times w}(q^{-m}, q^{-m+1},\ldots, , q^{n-1})
%        = \lim_{m \to \infty} q^{-mp}\fkS_{1^m \times w}(1, q, \ldots, q^{m+n-1})\\
%        &=  \sum_{a_1a_2\cdots a_p \in \cR(w)} \lim_{m \to \infty}  \tfrac{1}{[p]_q!} \cdot \tfrac{[a_1+m]_q}{q^m}\cdot \tfrac{[a_2+m]_q}{q^m} \cdots \tfrac{[a_p+m]_q}{q^m}\cdot q^{\comaj(a)}
%    \end{align*}
%    by Theorem~\ref{fs-thm}, where all limits are in the sense of formal power series.
%    The desired result holds since $\lim_{m \to \infty} [a_i+m]_q/q^m = \lim_{m \to \infty} (q^{a_1+m-1} + \cdots + q + 1)/q^m = q^{a_i-1} + q^{a_i-2} + q^{a_i-3} + \cdots$.
%%    Now observe that
%%$
%%        \lim_{m \to \infty} \frac{[a_i+m]_q}{q^m} = \lim_{m \to \infty} \frac{q^{a_1+m-1} + \cdots + q + 1}{q^m} = q^{a_i-1} + q^{a_i-2} + q^{a_i-3} + \cdots,
%%$
%%    which is the expansion of $q^{a_i}/(q-1)$ as a Laurent series in $q^{-1}$. So
%%    \begin{align*}
%%        A_w = \tfrac{1}{[\ell]_q!} \sum_{a \in \cR(w)} q^{\comaj(a)} \prod_{i=1}^{\ell} \tfrac{q^{a_i}}{q-1} = \tfrac{1}{[\ell]_q!(q-1)^{\ell}} \sum_{a \in \cR(w)} q^{\comaj(a)+a_1+\cdots+a_\ell},
%%    \end{align*}
%%    and $[\ell]_q!(q-1)^{\ell} = \prod_{i=1}^{\ell} (q^i-1)$.
%\end{proof}

\begin{example}\label{ex1}
Setting $x_i = q^{i-1}$ in
the definition of $\bS_w$ gives another formula for $\bS_w(x_i \mapsto q^{i-1}) $ as a sum 
over the reduced words for $w$.
%\[
%\bS_w(x_i \mapsto q^{i-1}) =
%\sum_{a \in \cR(w)} \sum_{\bfi \in \CS(a)} q^{(i_1-1)+(i_2-1)+ \dots+(i_{p}- 1)} .
%\]
The corresponding terms in these two summations need not agree, however: for a given word $a=a_1a_2\cdots a_p \in \cR(w)$, it can happen that
\[
 \sum_{\bfi \in \CS(a) } q^{(i_1-1)+(i_2-1)+ \dots+(i_{p}- 1)} 
  \neq
  \tfrac{q^{\sum a+\comaj(a)}}{(q-1)(q^{2}-1)\cdots (q^{p}-1)}.
%q^{\comaj(a)}\prod_{i=1}^{\ell(w)} q^{a_i} (q^{-i} + q^{-2i} + q^{-3i}+\dots)
  \]
For example, if $w= (1,2)(3,4)$ and $a=a_1a_2=1,3$ then
$ \sum_{\bfi \in \CS(a)} q^{(i_1-1)+ \dots+(i_{p}- 1)} $ is  
\[ \sum_{1 \geq i_1 <i_2 \leq 3} q^{(i_1-1)+(i_2-1)}  \in q^2 + 2q + 2 + q^{-1}\ZZ[[q^{-1}]] \]
while
$ \tfrac{q^{\sum a+\comaj(a)}}{(q-1)(q^{2}-1)\cdots (q^{p}-1)}= \frac{q^5}{(q-1)(q^2-1)}$ expands into the Laurent series
\[
q^{5} (q^{-1} + q^{-2} + q^{-3}+\dots) (q^{-2} + q^{-4} +\dots)
 \in q^2 + q + 2 + q^{-1}\ZZ[[q^{-1}]]. \]
 For $w = (1,2)(3,4)$ there are only two reduced words and one has 
 \[ \ba
 \bS_{(1,2)(3,4)} = \belem_1^2 &+ (2x_1 + x_2 + x_3)\belem_1 + x_1^2 + x_1 x_2 + x_1 x_3 
 \\
=  \dots &+ x_0^2 + 2x_{-1}x_1  +x_{-2} x_2  +x_{-3} x_3 
 \\& +  2x_0 x_1 + x_{-1}x_2  + x_{-2}x_3 
 \\& +x_1^2 + x_0 x_2 + x_{-1}x_3 
 \\ & + x_1 x_2 + x_0 x_3  
 \\ &+ x_1 x_3
 %\\ = e_1^2 &+ (2x_1 + x_2 + x_3)e_1 + x_1^2 + x_1 x_2 + x_1 x_3,
 \ea\]
 where $\belem_d$ is the elementary symmetric function $\sum_{i_1 <i_2< \cdots < i_d \leq 0} x_{i_1}x_{i_2} \cdots x_{i_d}$. One computes 
\[
  \bS_{(1,2)(3,4)}(x_i \mapsto q^{i-1}) = \tfrac{q^4}{(q-1)^2} = \dots + 7q^{-4} + 6q^{-3} + 5q^{-2} + 4q^{-1} + 3 + 2q + q^2
\]
using either Theorem~\ref{intro-thm-A} or the formula $\belem_d(q^{-1}, q^{-2}, \ldots) = \frac{1}{(q-1)(q^2-1)\cdots (q^d-1)}$.
  \end{example}

Our first new results are versions of the preceding theorem for Schubert polynomials in other classical types.
We begin with type B/C.
Given $0 < i <n$, define $t_i = t_{-i} := (i,i+1)(-i,-i-1)$ and $t_0 := (-1,1)$.
Define $\WBC_n :=\langle t_0,t_1,\dots,t_{n-1}\rangle$ to be the Coxeter group consisting of the permutations $w$ of $\ZZ$ with $w(i) = i$ for $|i| > n$ and $w(-i)=-w(i)$ for all $i \in \ZZ$.

A \emph{signed reduced word of type B} for an element $w \in \WBC_n$ is a word $a_1a_2\cdots a_p$ with letters in 
the set
$
\{ -n+1,\dots,-1,0,1,\dots,n-1\}
$ of shortest possible length such that 
$w = t_{a_1} t_{a_2}\cdots t_{a_p}$.
Let $-0$ denote a formal symbol distinct from $0$ that satisfies $-1<-0<0<1$ and set
$t_{-0} := t_0$.
A \emph{signed reduced word of type C} for $w \in \WBC_n$ is a word $a_1a_2\cdots a_p$ with letters in 
$
\{ -n+1,\dots,-1,-0,0,1,\dots,n-1\}
$ of shortest possible length such that 
$w = t_{a_1} t_{a_2}\cdots t_{a_p}$.
Let $\cR_B^\pm(w)$ and  $\cR_C^\pm(w)$ denote the respective sets of signed reduced words  for
 $w$.

\begin{definition}
The \emph{type B/C Schubert polynomials} of $w \in \WBC_n$ are
\[ \bSB_w := 
\sum_{\substack{a \in \cR_B^\pm(w)
\\
\bfi \in \CS(a)}}
x_\bfi
\qquand
 \bSC_w :=
\sum_{\substack{a \in \cR_C^\pm(w)
\\
\bfi \in \CS(a)}}
x_{\bfi}
=2^{\ell_0(w)} \bSB_w
\]
where $\ell_0(w) := |\{ i \in \ZZ : w(i) < 0<i\}|$.
\end{definition}

Both of the ``polynomials'' $\bSB_w$ and $\bSC_w$ are formal power series in $ \ZZ[[\dots,x_{-1},x_0,x_1,\dots,x_{n-1}]]$.
If we substitute $x_i\mapsto z_i$ for $i>0$ and $x_i \mapsto x_{1-i}$ for $i\leq 0$, then
 $\bSB_w$ and $\bSC_w$ specialize to the Schubert polynomials of types B and C
defined by Billey and Haiman in \cite{BH}; compare our definition with \cite[Thm. 3]{BH}.

Let $\cR_C(w)$ for $ w\in \WBC_n$ denote the subset of words in $\cR_C^\pm(w)$ whose letters all belong to $\{0,1,\dots,n-1\}$.
In Section~\ref{type-c-sect} we prove the following analogue of Theorem~\ref{intro-thm-A}.

\begin{theorem}\label{intro-thm-C}
 If $w \in \WBC_n$ then 
\[\bSC_w(x_i \mapsto q^{i-1}) =   \sum_{a=a_1a_2\cdots a_p \in \cR_C(w)} 
\tfrac{(q^{a_1}+1)(q^{a_2}+1)\cdots (q^{a_p}+1)}{(q-1)(q^{2}-1)\cdots (q^{p}-1)} q^{\comaj(a)}\]
where the right hand expression is interpreted as a Laurent series in $q^{-1}$.
\end{theorem}

\begin{example}
If $w = (1,-2)(2,-1) \in\WBC_n$ then the set $\cR_C^\pm(w)$ has 8 elements,
 formed by adding arbitrary signs to the letters in $a_1a_2a_3  =0,1,0$. One can compute that
 \[ \ba
 \bSC_{ (1,-2)(2,-1)} &= 4\belem_2 \belem_1 - 4\belem_3
 \\
 &= \dots +   4x_{-2}x_{-1}^2 + 4x_{-2}^2 x_0 + 8x_{-3}x_{-1}x_0 + 4x_{-4}x_0^2  
 \\& \hspace{9.6mm} +4x_{-3}x_0^2  +8x_{-2} x_{-1}x_0   \\ & \hspace{9.6mm} + 4x_{-2} x_0^2 + 4 x_{-1}^2x_0  \\ &  \hspace{9.6mm} + 4x_{-1} x_0^2 
 \ea\]
 where $\belem_d :=\sum_{i_1 <i_2< \cdots < i_d \leq 0} x_{i_1}x_{i_2} \cdots x_{i_d}$ as in Example~\ref{ex1}.
 It follows that
 \[
  \bSC_{(1,-2)(2,-1)}(x_i \mapsto q^{i-1}) = \tfrac{4q}{(q-1)^2(q^3-1)} = \dots + 36q^{-9} + 28q^{-8} + 20q^{-7} + 12q^{-6} + 8q^{-5} + 4q^{-4}.
  \]
  \end{example}

We turn next to type D.
For $1 < i < n$, let $r_i =r_{-i} := (i,i+1)(-i,-i-1)=t_i$ but define 
\[ r_1 := (1,2)(-1,-2) = t_1 \qquand r_{-1} := (1,-2)(-1,2) = t_0t_1t_0.\]
Define $\WD_n:=\langle r_{-1}, r_1,r_2,\dots,r_{n-1}\rangle$ to be the Coxeter group of permutations $w \in \WBC_n$ 
for which the number of positive integers $i$ with $w(i) < 0$ is even.
A \emph{signed reduced word} for $w \in \WD_n$
is a word $a_1a_2\cdots a_p$ with letters  in the set
 $\{-n+1,\dots,-2,-1,1,2,\dots,n-1\}$
 of shortest possible length such that 
$w = r_{a_1} r_{a_2}\cdots r_{a_p}$.
Let $\cR_D^\pm(w)$ denote the set of such words.

\begin{definition}
The \emph{type D Schubert polynomial} of $w \in \WD_n$ is
\[ \bSD_w  =
\sum_{a \in \cR_D^\pm(w)} 
\sum_{\bfi \in \CS(a)}
x_{\bfi}  \in  \ZZ[[\dots,x_{-1},x_0,x_1,\dots,x_{n-1}]].\]
\end{definition}

If we again substitute $x_i\mapsto z_i$ for $i>0$ and $x_i \mapsto x_{1-i}$ for $i\leq 0$, then
our definition of the power series $\bSD_w$ specializes to Billey and Haiman's formula for the Schubert polynomial of type D
given in \cite[Thm. 4]{BH}.

Suppose $a=a_1a_2\cdots a_p$ is a sequence of integers $a_i \in \{\pm 1, \pm 2, \pm3,\dots, \pm (n-1) \}$.
Define \be\exmaj(a) := |\{ i  : a_i >0\}| + \sum_{a_i \prec a_{i+1}} 2i\ee
where $\prec$ is the order $-1 \prec -2 \prec \dots \prec -n \prec1 \prec 2 \prec \dots \prec n$.
For example, if $a=a_1a_2a_3a_4 = -1,-2,3,1$ then $\exmaj(a) = 2 + (2 + 4)=8$.
We prove the following in Section~\ref{type-d-sect}.

\begin{theorem}\label{intro-thm-D}
 If $w \in \WD_n$ then 
\[ \bSD_w(x_i \mapsto q^{i-1}) =    \sum_{a=a_1a_2\cdots a_p \in \cR_D^\pm(w)}
\tfrac{(q^{|a_1|}+1)(q^{|a_2|}+1)\cdots (q^{|a_p|}+1)}{(q^2-1)(q^{4}-1)\cdots (q^{2p}-1)} q^{\exmaj(a)}\]
where the right hand expression is interpreted as a Laurent series in $q^{-1}$.
\end{theorem}

\begin{example}
If $w =(1,-1)(4,-4) \in\WD_n$ then the set $\cR_D^\pm(w)$ has 32 elements,
 formed by adding signs to the letters in $a_1a_2a_3a_4a_5a_6  =3,2,1,1,2,3$ in all ways that give opposite signs to the two entries 
 with absolute value one. One can compute that
 \[ \ba
 \bSD_{ (1,-1)(4,-4)}  &= x_1 x_2 x_3 \bP_3 + (x_1 x_2 + x_1 x_3 + x_2 x_3)\bP_4 + (x_1 + x_2 + x_3)\bP_5 + \bP_6 
 \\
& = \dots +   
 x_0^4 x_1 x_3 +  2 x_{-1}x_0^3 x_2 x_3 + 2x_{-1}^2x_0 x_1 x_2 x_3  + 2x_{-2}x_0^2 x_1 x_2 x_3   \\
 &\hspace{9.6mm} + x_0^4 x_2 x_3 + 2x_{-1}x_0^2 x_1 x_2 x_3   \\
 &\hspace{9.6mm} + x_0^3 x_1 x_2 x_3
\ea
\]
where $\bP_d$ for $d > 0$ is the Schur $P$-function $ \frac{1}{2}\sum_{a=0}^d e_a(x_0,x_{-1},\dots) h_{d-a}(x_0,x_{-1},\dots)$. Using
the formula 
$\bP_d(q^{-1}, q^{-2}, \ldots)  = \tfrac{(q+1)(q^2+1)\cdots (q^{d-1}+1)}{(q-1)(q^2-1)\cdots (q^d-1)}$
one can check that
\begin{equation*}
  \bSD_{(1,-1)(4,-4)}(x_i \mapsto q^{i-1}) = \tfrac{q^{12}(q^2+1)}{(q-1)^3(q^3-1)^2(q^5-1)} = \dots +  46q^{-5} + 27q^{-4} + 15q^{-3} + 7q^{-2} + 3q^{-1} + 1,
\end{equation*}
which agrees with Theorem~\ref{intro-thm-D}.
  \end{example}

Setting $q = 1$ in Theorem~\ref{intro-thm-A} leads to surprising enumerative formulas involving reduced words, compatible sequences, and plane partitions \cite{FominKirillov}. By contrast, the power series $\bS_w$, $\bSB_w$, $\bSC_w$, and $\bSD_w$ do not converge upon specializing $x_i \mapsto 1$ for all $i$. It would be interesting to find variations of our formulas with clearer enumerative content.

The second half of this note contains a few other related results.
In Section~\ref{groth-sect}, we extend Theorems~\ref{intro-thm-A}, \ref{intro-thm-C}, and \ref{intro-thm-D} to identities for \emph{Grothendieck polynomials}.
Our proofs of these formulas are fairly straightforward adaptations of the 
algebraic methods in \cite{FominStanley,KirillovNaruse}. It is an interesting open problem to find bijective proofs of these identities along the lines of \cite{BHY}.

Our approach has one other notable application, which we discuss in Section~\ref{last-sect}.
There, we develop a simple alternate proof of 
the main result of \cite{HMP6}, which gives a pipe dream formula for certain \emph{involution Schubert polynomials}.
In fact, we are able to prove a more general $K$-theoretic formula, partially resolving an open question from \cite[\S6]{HMP6}.

\subsection*{Acknowledgements}

The first author was partially supported by Hong Kong RGC Grant ECS 26305218.
We thank Sergey Fomin for suggesting the problem of finding analogues of Macdonald's formulas for Schubert polynomials outside type A. 

\section{Principal specializations of Schubert polynomials}

This section contains our proofs of Theorems~\ref{intro-thm-C} and \ref{intro-thm-D}.
Throughout, we fix a positive integer $n$ and let $R$ be an arbitrary commutative ring containing 
the ring of formal power series $  \ZZ[[x_i : i < n]].$

\subsection{Nil-Coxeter algebras}

The algebra introduced in this section 
figures prominently in \cite{FominStanley} and in several of our arguments.
Let $(W,S)$ be a Coxeter system with length function $\ell$. Let 
 $\cN=\cN(W)$ be the $R$-module of all formal $R$-linear combinations of the symbols $u_w$ for $w \in W$.
 This module
has a unique $R$-algebra structure 
with bilinear multiplication satisfying 
\[ u_v u_w = \begin{cases} u_{vw} &\text{if }\ell(vw) = \ell(v) + \ell(w) \\ 0 &\text{if }\ell(vw) < \ell(v) + \ell(w)\end{cases}
\qquad\text{for $v,w \in W$.}
\]
Following \cite[\S2]{FominStanley}, we refer to $\cN$ as the \emph{nil-Coxeter algebra} of $(W,S)$.
Choose $x,y\in R$.
Given $s \in S$, define $h_s(x) := 1 + x u_s \in\cN.$
One checks that if $s,t \in S$ and $st=ts$ then 
\[ h_s(x)h_s(y) = h_s(x+y)\qquand  h_s(x)h_t(y) = h_t(y)h_s(x).\]
We will also need the following general identity, which is equivalent to \cite[Lem. 5.4]{FominStanley} after some minor changes of variables:

   \begin{lemma}[{\cite[Lem. 5.4]{FominStanley}}]\label{formal-lem}
 Let $\fkt_1,\fkt_2,\dots,\fkt_{N}$ be some elements of an $R$-algebra with identity $1$,
 and suppose $q$, $z_1$, $z_2$, \dots $z_{N}$ are formal variables. Then
 \[
 \prod_{j=-\infty}^0 \prod_{i=1}^{N} (1 + q^{j-1} z_i \fkt_i) = 
 \sum_{p\geq 0} 
 \sum_{a_1,a_2,\dots,a_p}
\tfrac{z_{a_1}z_{a_2}\cdots z_{a_p}}{(q-1)(q^{2}-1)\cdots (q^{p}-1)}  q^{\comaj(a)}\fkt_{a_1} \fkt_{a_2}\cdots \fkt_{a_p}  
 \]
 where $\comaj(a) := \sum_{a_i< a_{i+1}} i$ and 
 the coefficients on the right are viewed as Laurent series in $q^{-1}$.
 \end{lemma}

\subsection{Type B/C}\label{type-c-sect}

Here, let $\cN=\cN(\WBC_n)$ denote the  nil-Coxeter algebra of  type B/C Coxeter system
$(W,S) = (\WBC_n, \{t_0,t_1,\dots,t_{n-1}\})$
and define  
$h_i(x) :=1 + x u_{t_i} \in \cN$ for integers $-n < i <n$ and $x \in R$.
Recall that $t_i = t_{-i}$ so we always have $h_i(x) = h_{-i}(x)$.
Let
\be\label{ABC-eq}
\ba 
A_i(x) & := h_{n-1}(x)h_{n-2}(x)\cdots h_i(x), \\
B(x) &:= h_{n-1}(x)\cdots h_1(x)h_0(x)h_{-1}(x) \cdots h_{-n+1}(x), \\
C(x) &:= h_{n-1}(x)\cdots h_1(x)h_0(x)h_0(x)h_{-1}(x) \cdots h_{-n+1}(x),
\ea
\ee
and note that $h_0(x)h_0(x) = h_0(2x)$.
Finally consider the infinite products in $\cN$ given by 
\be\label{bSBC-def}
\bSB := \prod_{i=-\infty}^0 B(x_i) \prod_{i=1}^{n-1} A_i(x_i) %= \cdots B(x_{-1})B(x_0) A_1(x_1) A_2(x_2)\cdots A_{n-1}(x_{n-1}).\]
\qquand
 \bSC := \prod_{i=-\infty}^0 C(x_i) \prod_{i=1}^{n-1} A_i(x_i). % = \cdots C(x_{-1})C(x_0) A_1(x_1) A_2(x_2)\cdots A_{n-1}(x_{n-1}).\]
\ee
It straightforward to see that
$
\bSB = \sum_{w \in \WBC_n} \bSB_w\cdot u_w
$
and
$ \bSC = \sum_{w \in \WBC_n} \bSC_w \cdot u_w 
$.
Less trivially:

%Given $i \in [n-1]$ and $x \in R$, define $\widetilde A_i(x) = h_i(x) h_{i+1}(x) \cdots h_{n-1}(x)$.

%\begin{lemma}[{\cite[Lem. 4.1]{FominStanley}}] 
%\label{commute-lem}
%The elements $A_i(x)$ and $ \widetilde A_i(y)$ commute in $\cN$.
%\end{lemma}

%\begin{lemma}[{\cite{FominStanley}}] 
%\label{commute-lem2}
%The elements $A_i(x)$ and $ \widetilde A_i(y)$ commute in $\cN$.
%Moreover,
%for any integer $1\leq i < n$ and  $x_i,\dots,x_{n-1},y \in R$ it holds that 
%\[\widetilde{A}_i(y) A_i(x_{i})A_{i+1}(x_{i+1})\cdots A_{n-1}(x_{n-1}) = \prod_{j=i+1}^{n-1} A_{j}(x_{j-1}) \cdot \prod_{j=i}^{n-1} h_j(x_j+y). \]
%\end{lemma}

%\begin{proof}
%The first claim is \cite[Lem. 4.1]{FominStanley}.
%It suffices to prove the second identity when $i=1$,
%and this follows by inverting both sides of \cite[Lem. 4.2]{FominStanley}.
%%Using Lemma~\ref{commute-lem} and the fact that $\widetilde A_i(x) = h_i(x) \widetilde A_{i+1}(x)$, one checks that 
%%$
%%\widetilde A_1(x_0) A_1(x_1)A_2(x_2)\cdots A_{n-1}(x_{n-1}) $
%%is equal to 
%%\[
%%A_1(x_1)h_1(x_0)\cdot  A_2(x_2) h_2(x_0) \cdot  A_3(x_3) h_3(x_0) \cdots A_{n-1}(x_{n-1})h_{n-1}(x_0)\]
%%which can be rewritten as 
%%\[
%%A_2(x_1)h_1(x_1 + x_0)\cdot  A_3(x_2) h_2(x_2 + x_0)  \cdots A_{n-1}(x_{n-2})h_{n-1}(x_{n-1}+x_0).\]
%%As $h_i(x)$ and $A_{i+2}(y)$ commute for each $i$, this expression
%%is in turn equal to
%%\[A_2(x_1)A_3(x_2)\cdots A_{n-1}(x_{n-2}) \cdot h_1(x_1 + x_0) h_2(x_2 + x_0)  \cdots h_{n-1}(x_{n-1}+x_0)\]
%%as desired.
%\end{proof}

\begin{proposition}
\label{prop2}
 It holds that 
\[\bSB = \prod_{j=-\infty}^0 \(h_0(x_j) \prod_{i=1}^{n-1} h_i(x_{i+j}+x_j)\)
\quand
\bSC = \prod_{j=-\infty}^0  \prod_{i=0}^{n-1} h_i(x_{i+j}+x_j).\]
\end{proposition}

\begin{proof}
We will just prove the formula for $\bSC$ since the other case is similar.
Let
\[\widetilde A_i(x) := h_i(x) h_{i+1}(x) \cdots h_{n-1}(x).\]
Since $A_i(x) = A_{i+1}(x) h_i(x)$ and $C(x) = A_1(x) h_0(x+x) \widetilde A_1(x)$,
 we have
\[
        \bSC =\prod_{i=-\infty}^{-1} C(x_i) \cdot A_1(x_0)h_0(x_0+x_0) \widetilde A_1(x_0) A_1(x_1)A_2(x_2)\cdots A_{n-1}(x_{n-1}).
\]
The elements $h_{i-2}(x)$, $A_i(y)$, and $\widetilde A_i(z)$ all commute by \cite[Lem. 4.1]{FominStanley}.
Using this fact and the identities $A_i(x) = A_{i+1}(x) h_i(x)$ and $\widetilde A_i(x) =  h_i(x)\widetilde A_{i+1}(x)$, %and $h_i(x)h_i(y) = h_i(x+y)$, 
it is easy
to show that 
\[
h_0(x_0+x_0)\widetilde A_1(x_0) A_1(x_1)A_2(x_2)\cdots A_{n-1}(x_{n-1})
=
\prod_{i=2}^{n-1} A_i(x_{i-1})\cdot \prod_{i=0}^{n-1} h_i(x_{i}+x_0).
\] 
Substituting this into our formula above gives
$ \bSC = \bSC(x_i \mapsto x_{i-1})  \prod_{i=0}^{n-1} h_i(x_i+ x_0)$
so by induction we have
$       \bSC = \bSC(x_i \mapsto x_{i-N})  \prod_{j=-N+1}^{0} \prod_{i=0}^{n-1} h_i(x_{i+j}+x_j)$
for all $N \geq 0$. But it is easy to see that  $\lim_{N \to \infty} \bSC(x_i \mapsto x_{i-N}) = 1$ as a limit of power series,
 so the result follows by sending $N \to \infty$.
\end{proof}

%The following alternate compatible sequence formula for $\bSC_w$
%is an immediate corollary of Proposition~\ref{prop2}.
%
%\begin{corollary} \label{cor:backstable-monomials-c}
%    For each $w \in \WBC_n$, it holds that 
%\[
%        \bSC_w = \sum_{a_1 \cdots a_p \in \cR(w)} \sum_{\substack{i_1 \leq \cdots \leq i_p \leq 0 \\ i_j < i_{j+1} \text{ if }a_j > a_{j+1} }} (x_{a_1+i_1} + x_{i_1}) \cdots (x_{a_p+i_p}+ x_{i_p}).
%\]
%\end{corollary}

 We can now prove Theorem~\ref{intro-thm-C}.

\begin{proof}[Proof of Theorem~\ref{intro-thm-C}]
To obtain the desired formula, set $x_i =q^{i-1}$ in Proposition~\ref{prop2}, apply Lemma~\ref{formal-lem}  
with $N=n$, $z_i =1+q^{i-1}$, and $\fkt_i = u_{t_{i-1}}$, and then extract the coefficient of $u_w$.
\end{proof}

\subsection{Type D}\label{type-d-sect}

Now, let $\cN=\cN(\WD_n)$ denote the  nil-Coxeter algebra of $(W,S) = (\WD_n, \{r_{-1},r_1,\dots,r_{n-1}\})$
and define  
$h_i(x) :=1 + x u_{t_i} \in \cN$ for all $i \in \{ \pm 1, \pm 2, \dots, \pm(n-1)\}$ and $x \in R$.
Let
\be\label{AAD-eq}
\ba 
A_i(x) &: = h_{n-1}(x)h_{n-2}(x)\cdots h_i(x), \\
\widetilde A_i(x) & := h_i(x) h_{i+1}(x) \cdots h_{n-1}(x), \\
D(x) &:= h_{n-1}(x)\cdots h_1(x)h_{-1}(x) \cdots h_{-n+1}(x).
\ea
\ee
The Coxeter group $\WD_n$ has a unique automorphism $w\mapsto w^*$
that maps $r_i \mapsto r_{-i}$ for $1 \leq i < n$. This map extends by linearity to an $R$-algebra automorphism of $\cN$
with $u_w^* : = u_{w^*}$. 
% by 
%\[ \Bigl( \sum c_w u_w \Bigr)^* = \sum c_w u_{w^*}\qquad\text{for any coefficients }c_w \in R.\]
We have $A_i(x)^* = A_i(x)$ for $1<i<n$ and $D(x)^* =D(x)$, while
$A_1(x)^* = h_{n-1}(x)h_{n-2}(x)\cdots h_2(x) h_{-1}(x).$
Consider the infinite products in $\cN$ given by
\be\label{bSD-def}
\bSD := \prod_{i=-\infty}^0 D(x_i)  \prod_{i=1}^{n-1} A_i(x_i) 
\qquand \bSE := \prod_{i=-\infty}^0 D(x_i)   \prod_{i=1}^{n-1} A_i(x_i)^*.\ee
It is easy to see that
$
\bSD = \sum_{w \in \WD_n} \bSD_w\cdot u_w
$
and
$ 
\bSE = \sum_{w \in \WD_n} \bSD_w\cdot u_{w^*}.
$
In addition:

\begin{proposition}
\label{prop3}
It holds that 
\[
  \bSD = \prod_{j=-\infty}^0 \(\prod_{i=1}^{n-1} h_{-i} (x_{i + 2j - 1}+x_{2j-1}) \prod_{i=1}^{n-1} h_i(x_{i +2j}+x_{2j})\).
  \]
\end{proposition}

\begin{proof}
Since $A_i(x) = A_{i+1}(x) h_i(x)$ and $D(x) = A_1(x)^*  \widetilde A_1(x)$,
 we have
\[
        \bSD =\prod_{i=-\infty}^{-1} D(x_i) \cdot A_1(x_0)^* \widetilde A_1(x_0) A_1(x_1)A_2(x_2)\cdots A_{n-1}(x_{n-1}).
\]
Repeating  the argument in the proof of Proposition~\ref{prop2}, we
deduce that
%the expression $\widetilde A_1(x_0) \cdot A_1(x_1)A_2(x_2)\cdots A_{n-1}(x_{n-1})$ is equal to 
%\[A_2(x_1)^*A_3(x_2)^*\cdots A_{n-1}(x_{n-2})^* \cdot h_1(x_1 + x_0) h_2(x_2 + x_0)  \cdots h_{n-1}(x_{n-1}+x_0).\]
%Substituting this 
% into our equation for $\bSD$ above gives 
$ \bSD = \bSE(x_i \mapsto x_{i-1})  \prod_{i=1}^{n-1} h_i(x_i+ x_0).$
An analogous identity holds for $\bSE$. 
Alternating these formulas gives
\[\bSD = \bSD(x_i \mapsto x_{i-2N})  \prod_{j=-N+1}^{0} \(\prod_{i=1}^{n-1} h_{-i}(x_{i+2j-1}+x_{2j-1})\prod_{i=1}^{n-1} h_i(x_{i+2j}+x_{2j})\)\]
for all $N \geq 0$. It is again easy to see that $\lim_{N \to \infty} \bSD(x_i \mapsto x_{i-2N}) = 1$ as a limit of formal power series,
so the result follows by sending $N \to \infty$.
\end{proof}

 We can now also prove Theorem~\ref{intro-thm-D}.

\begin{proof}[Proof of Theorem~\ref{intro-thm-D}]
 By Proposition~\ref{prop3} we have \[
 \bSD(x_i \mapsto q^{i-1}) = \prod_{j=-\infty}^0 \( \prod_{i=1}^{n-1} \(1 + q^{2(j-1)} \cdot (1+q^i) \cdot u_{r_{-i}}\)\cdot  \prod_{i=1}^{n-1} \(1 + q^{2(j-1)} \cdot q(1+q^i) \cdot u_{r_{i}}\)\).  \]
 To get the desired expression for $\bSD_w$, apply Lemma~\ref{formal-lem} with $q$ replaced by $q^2$ and $N=2n-2$ to the right side
 of the preceding identity,
using the parameters $z_i =1+q^{i}$, $ z_{n-1+i} =q(1+q^{i})$, $\fkt_i = u_{r_{-i}}$, and $\fkt_{n-1+i}= u_{r_{i}}$ for $1\leq i <n$.
Then extract the coefficient of $u_w$.
\end{proof}

\section{Principal specializations of Grothendieck polynomials}\label{groth-sect}

In this section we describe some extensions of Theorems~\ref{intro-thm-A}, \ref{intro-thm-C}, and \ref{intro-thm-D}
for \emph{Grothendieck polynomials} in classical types.
The identities proved here are more general but also more technical than the formulas sketched in the introduction.

\subsection{Id-Coxeter algebras}

Again let $(W,S)$ be an arbitrary Coxeter system with length function $\ell$. 
For the results in this section, we work in a generalization of 
the algebra $\cN(W)$.
Recall that $R$ is an arbitrary commutative ring containing $  \ZZ[[x_i : i < n]].$
From this point on, we fix an element $\beta \in R$.

Let $\cH = \cH(W)$ be the $R$-module of all formal $R$-linear combinations of the symbols $\pi_w$ for $w \in W$.
This module  has a unique $R$-algebra structure with bilinear multiplication satisfying 
\[ \pi_v \pi_w = \pi_{vw}\text{ if $\ell(vw)=\ell(v)+\ell(w)$}
\qquand
\pi_s^2 = \beta \pi_s
\]
for $v,w \in W$ and $s \in S$
 \cite[Def. 1]{KirillovNaruse}, which we refer to as
the \emph{id-Coxeter algebra} of $(W,S)$.
%
%Evidently $\cH = \cN$ if $\beta =0$.
For $x,y\in R$ and $s \in S$, define 
\be x \oplus y := x + y + \beta xy
\qquand \bh_s(x) := 1 + x \pi_s.\ee
% \qquand
%x \ominus y = \tfrac{x-y}{1+\beta y}
%\] so that $(x \oplus y) \ominus y = x \oplus (y \ominus y) = x$.
Then
$ \bh_s(x)\bh_s(y) = \bh_s(x\oplus y)$, and
if $st=ts$
then
 $ \bh_s(x)\bh_t(y) = \bh_t(y)\bh_s(x)$  \cite[Lem. 1]{KirillovNaruse}.

%We note some identities.
%Let $s,t \in S$. If $st=ts$ then 
% \be h_s(x)h_s(y) = h_s(x\oplus y) \qquand h_s(x)h_t(y) = h_t(y)h_s(x).\ee
%If $st$ has order 3 so that $sts=tst$, then 
% \be h_s(x)h_t(x\oplus y) h_s(y) = h_t(y)h_s(x\oplus y) h_t(x).\ee
% Finally, if $st$ has order 4 so that $stst = tsts$, then 
% \be h_s(x\ominus y)h_t(x) h_s(x\oplus y)h_s(y) = h_s(y) h_t(x\oplus y) h_s(x) h_t(x\ominus y).\ee
% These identities appear, for example, as \cite[Lemma 1]{KirillovNaruse}, are easy to check directly.
 
\subsection{Type A} 

Let $\bSym_n := \langle s_i : i <n\rangle$ be the Coxeter group of permutations $w \in S_\ZZ$ with $w(i)=i$ for all $i>n$.
In this section we write $\cH=\cH(\bSym_n)$ 
and set
$\pi_i := \pi_{s_i}\in\cH $ for integers $i<n$.
Define $\Hecke(w)$ for $w \in \bSym_n$ to be the set of words $a_1a_2\cdots a_N$ %with letters in $\{ i \in \ZZ : i < n\}$
such that  
$ \pi_w = \beta^{N-\ell(w)} \pi_{a_1} \pi_{a_2}\cdots \pi_{a_N}$.
Recall the set $\CS(a)$ from Definition~\ref{compat-def}.

\begin{definition}
The \emph{backstable Grothendieck polynomial} of $w \in S_n \subsetneq \bSym_n$ is
\[\bG_w := \sum_{a \in \Hecke(w)} \sum_{\bfi \in \CS(a)} \beta^{\ell(\bfi)-\ell(w)} x_\bfi \in  \ZZ[\beta][[\dots,x_{-1},x_0,x_1,\dots,x_{n-1}]].\]
\end{definition}

The function $\fkG_w := \bG_w(\dots,0,0,x_1,x_2,\dots,x_{n-1})$ is the ordinary \emph{Grothendieck polynomial} of $w \in S_n$.
The power series $G_w:=\bG_w(\dots,x_{3}, x_{2},x_1,0,0,\dots,0)$ 
given by setting $x_i \mapsto 0$ for $i>0$ and $x_i \mapsto x_{1-i}$ for $i \leq 0$
 is a symmetric function in the $x_i$ variables, which is usually called the \emph{stable Grothendieck polynomial} of $w \in S_n$.
 
 Specializing $\beta\mapsto 0$ transforms $\bG_w\mapsto \bS_w$ from Section~\ref{intro-sect}. 
The Grothendieck polynomials $\fkG_w$ are closely related to the $K$-theory of flag varieties and Grassmannians \cite{Buch, MP2019a}.
We do not know of a similar geometric interpretation for the backstable Grothendieck polynomials $\bG_w$.

For $i <n$ and $x \in R$, let 
$
\bh_i(x) :=  1 + x \pi_{i} 
$
and
$
 \bA_i(x): = \bh_{n-1}(x)\bh_{n-2}(x) \cdots \bh_i(x)$.
Define
\be\label{bG-def}\bG :=\cdots \bA_{n-3}(x_{n-3})\bA_{n-2}(x_{n-2})\bA_{n-1}(x_{n-1}) = \prod_{i=-\infty}^{n-1} \bA_i(x_i) \in \cH.\ee
If $w \in S_n$ then the coefficient of $\pi_w$ in this expression is $\bG_w$.
%$\bG = \sum_{w \in \bSym_n} \bG_w \cdot \pi_w.$
%We derive an alternate expression for $\bG$.

\begin{proposition}
\label{prop1}
 It holds that $\bG = \prod_{j=-\infty}^0 \prod_{i=-\infty}^{n-1} \bh_i(x_{i+j})$.
\end{proposition}

\begin{proof}
We have
$
        \bG% &= \cdots \bA_0(x_0)\bA_1(x_1)\cdots \bA_{n-1}(x_{n-1})
        = \cdots \bA_1(x_0)\bh_0(x_0) \bA_2(x_1) \bh_1(x_1)\cdots \bA_{n-1}(x_{n-2})\bh_{n-1}(x_{n-1})
$
by definition.
As $\bh_i(x)$ and $\bA_{i+2}(y)$ commute, it follows that
$\label{bS-id1} \bG = \bG(x_i \mapsto x_{i-1})  \prod_{i=-\infty}^{n-1}\bh_i(x_i)$
so by induction
$       \bG = \bG(x_i \mapsto x_{i-N})  \prod_{j=-N+1}^{0} \prod_{i=-\infty}^{n-1} \bh_i(x_{i+j})$
for all $N \geq 0$. But we have $\lim_{N \to \infty} \bG(x_i \mapsto x_{i-N}) = 1$ as a limit of formal power series,
so the result follows by sending $N \to \infty$.
\end{proof}

%As a corollary, we obtain an alternate compatible sequence formula for $\bG_w$.
%
%\begin{corollary} \label{cor:backstable-monomials}
%    For each $w \in S_n\subsetneq \bSym_n$, it holds that 
%\[
%        \bG_w = \sum_{a_1a_2 \cdots a_N \in \Hecke(w)}  \beta^{N-\ell(w)}\sum_{\substack{i_1 \leq i_2\leq \cdots \leq i_N \leq 0 \\ i_j < i_{j+1} \text{ if }a_j > a_{j+1} }} x_{a_1+i_1}x_{a_2+i_2} \cdots x_{a_N+i_N}.
%\]
%\end{corollary}

We can now prove the following generalization of Theorem~\ref{intro-thm-A}.

 \begin{theorem}\label{groth-thm}
If $w \in S_n \subsetneq \bSym_n$ then
$
\bG_w(x_i \mapsto q^{i-1}) = \sum_{a \in \Hecke(w)}  \tfrac{\beta^{\ell(a)-\ell(w)}}{(q-1)(q^{2}-1)\cdots (q^{\ell(a)}-1)}q^{\sum a+\comaj(a)}
$
where the right hand expression is interpreted as a Laurent series in $q^{-1}$.
\end{theorem}

\begin{proof}
If $w \in S_n$ then the coefficient of $\pi_w$ in $\bG$ is the same as the coefficient 
of $\pi_w$ in the truncated product $ \prod_{j=-\infty}^0 \prod_{i=1}^{n-1} \bh_i(x_{i+j})$.
This coefficient is  $\bG_w$, and the theorem follows by
applying Lemma~\ref{formal-lem} 
with $N=n-1$ and $z_i \fkt_i=q^i \pi_{s_i}$ to the latter expression.
\end{proof}

  There are \emph{Grothendieck polynomials} in the other classical types  \cite{KirillovNaruse}
 which generalize $\bSB_w$, $\bSC_w$, and $\bSD_w$ in the same way that $\bG_w$ generalizes $\bS_w$. 
We discuss these formal power series next.

\subsection{Type B/C}

In this section let $\cH=\cH(\WBC_n)$
 and  write $\pi_i := \pi_{t_i}\in\cH $ for $-n<i<n$.
Given  a permutation $w \in \WBC_n$, 
define $\Hecke_B^\pm(w)$ and  $\Hecke_C^\pm(w)$ to be the sets of  words $a_1a_2\cdots a_N$, with letters in
$ \{-n+1,\dots,-1,0,1,\dots,n-1\}$
and
$\{-n+1<\dots<-1<-0<0<1<\dots<n-1\}$, respectively,
such that  
$ \pi_w = \beta^{N-\ell(w)} \pi_{a_1} \pi_{a_2}\cdots \pi_{a_N} \in \cH$,
where $\ell(w)$ denotes the usual Coxeter length of $w$ and  $\pi_{-0} := \pi_0 \in \cH$.
Recall that we view $-0$ as a symbol distinct from $0$.

\begin{definition}
The \emph{type B/C Grothendieck polynomials} of $w \in \WBC_n$ are
\[\bGB_w := \sum_{\substack{a \in \Hecke_B^\pm(w) \\ \bfi \in \CS(a)}} \beta^{\ell(\bfi)-\ell(w)} x_\bfi
\quand
\bGC_w := \sum_{\substack{a \in \Hecke_C^\pm(w) \\ \bfi \in \CS(a)}} \beta^{\ell(\bfi)-\ell(w)} x_\bfi.
\]
\end{definition}

We may consider the finite sums
\[\ds
\bGB := \sum_{w \in \WBC_n} \bGB_w\cdot \pi_w \in \cH(\WBC_n)
\quand \ds \bGC := \sum_{w \in \WBC_n} \bGC_w \cdot \pi_w \in \cH(\WBC_n)
.\]
 Define $\bA_i(x)$, $\bB(x)$, and $\bC(x)$ as  in \eqref{ABC-eq} %as the ones defining $A_i(x),B(x), C(x) \in \cN(\WBC_n)$ 
but with $h_i (x)$ replaced by 
\[
\bh_i(x) := 1 + x \pi_i \in \cH(\WBC_n)\qquad\text{for $-n<i<n$ and $x \in R$}.
\]
Then  $\bGB$ and $\bGC$ are given by the formulas in \eqref{bSBC-def} 
with $A_i$, $B$, $C$ replaced by $\bA_i$, $\bB$, $\bC$.
Comparing with \cite[Def. 9]{KirillovNaruse} shows that 
$\bGB_w$ and $\bGC_w$ are obtained from
Kirillov and Naruse's \emph{double Grothendieck polynomials}  $\cG_w^{\textsf{B}}(a,b;x)$ and $\cG_w^{\textsf{C}}(a,b;x)$ 
by setting $a_i \mapsto x_{i}$, $b_i\mapsto 0$, and $x_i \mapsto x_{1-i}$.

 \begin{proposition}\label{groth-prop}
  It holds that %The following identities hold in  $\cH = \cH(\WBC_n)$:
\[
\bGB = 
   \prod_{j=-\infty}^0 \( \bh_0(x_j) \prod_{i=1}^{n-1}\bh_i(x_{i+j} \oplus x_j)\)
\qquand
\bGC =
   \prod_{j=-\infty}^0  \prod_{i=0}^{n-1} \bh_i(x_{i+j}\oplus x_j).
\]
\end{proposition}
 
 \begin{proof}

 Since $\bA_i(x)$ commutes with
${\widetilde A}^{(\beta)}_i(x) := \bh_i(x)\bh_{i+1}(x)\cdots \bh_{n-1}(x)$ by \cite[Lem. 3]{KirillovNaruse},
the result follows by the same proof as Proposition~\ref{prop2}, \emph{mutatis mutandis}.
 \end{proof}

%Let $\Hecke_C(w)$ for $w \in \WBC_n$ denote the set of words $a_1a_2\cdots a_p \in \Hecke_C^\pm(w)$ with letters $a_i \in \{0,1,\dots,n-1\}$.
%The following %alternate compatible sequence formula for $\bGC_w$
%is an immediate corollary of Proposition~\ref{groth-prop}.
%
%\begin{corollary} \label{cor:backstable-monomials-c}
%    For each $w \in \WBC_n$, it holds that 
%\[
%        \bGC_w = \sum_{a_1 \cdots a_N \in \Hecke_C(w)}   \beta^{N-\ell(w)}\sum_{\substack{i_1 \leq \cdots \leq i_N \leq 0 \\ i_j < i_{j+1} \text{ if }a_j > a_{j+1} }} (x_{a_1+i_1} \oplus x_{i_1}) \cdots (x_{a_N+i_N}\oplus x_{i_N}).
%\]
%\end{corollary}

Given a word $a=a_1a_2\cdots a_p$ with $a_i \in \{-n+1<\dots<-1<-0<0<1<\dots<n-1\}$,
let $I(a)$ be the set of indices $ i \in [p]$ with $a_i \in \{1,2,\dots,n-1\}$
and define
\be \SigmaC(a) := \sum_{i \in I(a)} a_i \qquand \comajc(a) :=   \sum_{a_i \prec a_{i+1}} i\ee
where $\prec$ is the order $-0\prec 0\prec -1\prec 1\prec -2\prec 2\prec \dots$.
For example, if $a = -1,1,-2,1$ then 
$\SigmaC(a) = 1 + 1 = 2$ and $\comajc(a) =1 + 2 = 3$.

\begin{theorem}\label{groth-thm-C}
If $w \in \WBC_n$  
then the following identities hold:
\ben
\item[(a)]
$\bGB_w(x_i \mapsto q^{i-1}) = \sum_{a \in \Hecke_B^\pm(w)} \frac{\beta^{\ell(a) - \ell(w)}}{(q-1)(q^{2}-1)\cdots (q^{\ell(a)}-1)} q^{ \SigmaC(a) + \comajc(a)}$.
\item[(b)]
$\bGC_w(x_i \mapsto q^{i-1}) = \sum_{a \in \Hecke_C^\pm(w)} \frac{\beta^{\ell(a) - \ell(w)}}{(q-1)(q^{2}-1)\cdots (q^{\ell(a)}-1)} q^{ \SigmaC(a) + \comajc(a)}$.
\een
The right hand expressions in both parts are interpreted as Laurent series in $q^{-1}$.
\end{theorem}

The second identity reduces to Theorem~\ref{intro-thm-C} when $\beta=0$
since the sum $\sum_a q^{ \SigmaC(a) + \comajc(a)}$ over all words $a=a_1a_2\cdots a_p \in \cR_C^\pm(w)$
with the same unsigned form is exactly the product  $(q^{|a_1|}+1)(q^{|a_2|}+1)\cdots (q^{|a_p|}+1) q^{\comaj(|a_1||a_2|\cdots |a_p|)}$.
\begin{proof}
Part (a) is similar so
we just prove (b).
As $\bh_i(x_{i+j}\oplus x_j) = \bh_i(x_j) \bh_i(x_{i+j})$,
we have
 \[
 \bGC(x_i \mapsto q^{i-1}) = \prod_{j=-\infty}^0\prod_{i=0}^{n-1} (1 + q^{j-1}\cdot \pi_i)( 1 + q^{j-1} \cdot q^i \cdot\pi_i)\]
 by Proposition~\ref{groth-prop}. The identity for $\bGC_w$
 follows by extracting the coefficient of $\pi_w$ from the right side after 
 applying Lemma~\ref{formal-lem} with $N=2n$ and with the parameters
 $z_1,z_2,\dots,z_{2n}$ and $\fkt_1,\fkt_2,\dots,\fkt_{2n}$ replaced by $1,1,1,q,1,q^2,\dots,1,q^{n-1}$
and $\pi_0,\pi_0,\pi_1,\pi_1,\dots,\pi_{n-1},\pi_{n-1}$, respectively.
\end{proof}

\subsection{Type D}

In this section let
$\cH=\cH(\WD_n)$
and $\pi_i := \pi_{r_i} \in \cH$. %for  $i \in [\pm (n-1)]$.
Given   
$w \in \WD_n$,
let $\Hecke_D^\pm(w)$ be the set of  words $a_1a_2\cdots a_N$ with letters in 
$[\pm(n-1)]:=\{\pm 1, \pm2,\dots, \pm (n-1)\}$
such that  
$ \pi_w = \beta^{N-\ell(w)} \pi_{a_1} \pi_{a_2}\cdots \pi_{a_N}\in \cH$, 
where $\ell(w)$ is the usual Coxeter length.

\begin{definition}
The \emph{type D Grothendieck polynomial} of $w \in \WD_n$ is
\[\bGD_w := \sum_{a \in \Hecke_D^\pm(w)} \sum_{\bfi \in \CS(a)} \beta^{\ell(\bfi)-\ell(w)} x_\bfi.\]
\end{definition}

We consider the sum
\[\bGD := \sum_{w \in \WD_n} \bGD_w \cdot \pi_w \in \cH(\WD_n).\]
 If we define $\bA_i(x)$  and $\bD(x)$ as  in \eqref{AAD-eq} %as the ones defining $A_i(x),B(x), C(x) \in \cN(\WBC_n)$ 
but with $h_i(x) $ replaced by
\[
\bh_i(x) := 1 + x \pi_i \in \cH(\WD_n)\qquad\text{for $i\in [\pm (n-1)]$ and $x \in R$},
\]
then  $\bGD$ is given by the formula in \eqref{bSD-def} 
with $A_i$ and $D$ replaced by $\bA_i$ and $\bD$.
Comparing with \cite[Def. 9]{KirillovNaruse} shows that 
$\bGD_w$ is obtained from
Kirillov and Naruse's \emph{double Grothendieck polynomial}  $\cG_w^{\textsf{D}}(a,b;x)$  
by making the substitutions $a_i \mapsto x_{i}$, $b_i\mapsto 0$, and $x_i \mapsto x_{1-i}$.

 \begin{proposition}\label{groth-prop-d} It holds that %The following identities hold in  $\cH = \cH(\WBC_n)$:
$\ds
\bGD =\prod_{j=-\infty}^0 \(\prod_{i=1}^{n-1} \bh_{-i} (x_{i + 2j - 1}\oplus x_{2j-1}) \prod_{i=1}^{n-1} \bh_i(x_{i +2j}\oplus x_{2j})\).
$
\end{proposition}

\begin{proof}
Similar to Proposition~\ref{groth-prop}, the result follows by repeating the proof of Proposition~\ref{prop3}
after adding ``$(\beta)$'' superscripts to all relevant symbols
 and
substituting $\bh_i(x)\bh_i(y) = \bh_i(x\oplus y)$ wherever the identity $h_i(x)h_i(y) = h_i(x+y)$ is used.
%and $\bA_i(x)$ commutes with
%${\widetilde A}^{(\beta)}_i(x) := \bh_i(x)\bh_{i+1}(x)\cdots \bh_{n-1}(x)$ by \cite[Lem. 3]{KirillovNaruse}
\end{proof}

%As a corollary we get another compatible sequence formula for $\bSD_w$.
%
%\begin{corollary} \label{cor:backstable-monomials-d}
%    For each $w \in \WD_n$, it holds that
%\[
%        \bGD_w = \sum_{a_1 \cdots a_N \in \Hecke_D^\pm(w)} \beta^{N-\ell(w)}
%        \sum_{i} (x_{|a_1|+i_1} \oplus x_{i_1}) \cdots (x_{|a_N|+i_N}\oplus x_{i_N})
%\]
%where the inner sum is over integer sequences $i=(i_1 \leq \dots \leq i_p \leq 0)$
%with $i_j$ odd if and only if $a_j<0$,
%and with
% $\lceil i_j /2 \rceil < \lceil i_{j+1} /2\rceil$ whenever $a_j$ and $a_{j+1}$ are out of order in the chain
%\[-1 \prec -2 \prec \dots \prec -n+1 \prec 1 \prec 2 \prec \dots \prec n-1.\]
%\end{corollary}

To state an analogue of Theorem~\ref{intro-thm-D} for $\bGD_w$,
we must consider the ordered alphabet 
\[ \{ -1' \prec -1 \prec -2' \prec -2 \prec \dots \prec -n' \prec -n \prec 1' \prec 1 \prec 2' \prec 2 \prec \dots \prec n' \prec n\}.\]
If $w \in \WD_n$ then let $\PrimedHecke_D^\pm(w)$ denote the set of words in this alphabet which become elements of $\Hecke_D^\pm(w)$
when all primes are removed from its letters.
Given such a word $a=a_1a_2\cdots a_p$, let $J(a)$ be the set of indices $i \in [p]$ for which $a_i$ is unprimed,
and define
\[
\SigmaD(a) := \sum_{i \in J(a)} |a_i|\qquand 
\exmaj(a) := |\{ i  : a_i \in \{1',1,2',2,\dots\}\}| + \sum_{a_i \prec a_{i+1}} 2i.
\]
For example, if $a = 2',-1',-1,-3,2$ then $\SigmaD(a) = 1 + 3+2 = 6$ and $\exmaj(a) = 2 + (4 + 6 + 8)=20$.

\begin{theorem}
If $w \in \WD_n$  
then
\[
\bGD_w(x_i \mapsto q^{i-1}) = \sum_{a \in \PrimedHecke_D^\pm(w)} \tfrac{\beta^{\ell(a) - \ell(w)}}{(q^2-1)(q^{4}-1)\cdots (q^{2\ell(a)}-1)} q^{\SigmaD(a)+\exmaj(a)}
\]
where the right hand expression is interpreted as a Laurent series in $q^{-1}$.
\end{theorem}

As with Theorem~\ref{groth-thm-C}, this identity reduces to Theorem~\ref{intro-thm-D} when $\beta=0$.

\begin{proof}
The proof is similar to Theorem~\ref{groth-thm-C}.
Proposition~\ref{groth-prop-d} implies that $ \bGD(x_i \mapsto q^{i-1}) $ is 
 \[
 \prod_{j=-\infty}^0
 \(
 \prod_{i=1}^{n-1} (1 + q^{2(j-1)} \cdot \pi_{-i})  (1 + q^{2(j-1)} \cdot q^i\cdot \pi_{-i})
 \prod_{i=1}^{n-1} 
  (1 + q^{2(j-1)} \cdot q\cdot \pi_{i})  (1 + q^{2(j-1)} \cdot q^{i+1}\cdot \pi_{i})
 \).
 \]
The identity for $\bGD_w$
 follows by extracting the coefficient of $\pi_w$ from this expression after 
 applying Lemma~\ref{formal-lem} with $q$ replaced by $q^2$ and with $N=4n-4$.
 When applying the lemma,
 we set the parameters
 $z_1,z_2,\dots,z_{2n-2}$ (respectively, $z_{2n-1},z_{2n},\dots,z_{4n-4}$)  to
 $1,q,1,q^2,1,q^3\dots$ (respectively, $q,q^2, q,q^3,q,q^4\dots$),
 while taking $\fkt_1,\fkt_2,\dots,\fkt_{2n-2}$ (respectively, $\fkt_{2n-1},\fkt_{2n},\dots,\fkt_{4n-4}$) to be
$\pi_{-1},\pi_{-1},\pi_{-2},\pi_{-2},\dots$ (respectively, $\pi_1,\pi_1,\pi_2,\pi_2,\dots$).
\end{proof}

 \section{Involution Grothendieck polynomials}
\label{last-sect}

This final section is something of a digression.
Here, we reuse the techniques introduced above to give a simple proof of a new formula 
for certain \emph{involution Grothendieck polynomials}.

In this section, 
we let $\cH=\cH(S_n)$ be the id-Coxeter algebra for the finite Coxeter system
$(W,S)=(S_n, \{s_1,s_2,\dots,s_{n-1}\})$,
and write $\pi_i := \pi_{s_i} \in \cH$.
Let 
\[\bI := \left\{w  \in S_n : w=w^{-1}\right\}
\qquand \bIfpf := \left\{ w^{-1} \bFPF w : w \in S_n\right\}\]
 where $\bFPF=\cdots(1,2)(3,4)(5,6)\cdots$ denotes the permutation of $\ZZ$ mapping $i \mapsto  i -(-1)^i.$
The sets $\bI$ and $\bIfpf$ are always disjoint, although when $n$ is even the elements of $\bIfpf$ are naturally
in bijection with the fixed-point-free elements of $\bI$.

Let $\IdInvol$ and $\IdFixed$ denote the free $R$-modules consisting of all $R$-linear combinations
of the symbols $m_z$  for  $z\in\bI$ and $z\in\bIfpf$, respectively.
These sets have unique right $\cH$-module structures (see \cite[\S1.2 and \S1.3]{Marberg2019a})
satisfying, for each integer $1\leq i<n$,
\[\ba
    m_z\pi_i &= \begin{cases} m_{zs_i} & \text{if $z(i) < z(i+1)$ and $zs_i = s_i z$}\\
        m_{s_i z s_i} & \text{if $z(i) < z(i+1)$ and $zs_i \neq s_i z$}\\
        \beta m_{z} & \text{if $z(i) > z(i+1)$}
    \end{cases}
    &&\quad\text{for $z \in \bI$}
\ea\]
and
\[\ba
    m_z\pi_i &= \begin{cases}
        m_{s_i z s_i} & \text{if $z(i) < z(i+1)$}\\
        \beta m_{z} & \text{if $i+1\neq z(i) > z(i+1)\neq i$}\\
      0 & \text{if $i+1= z(i) > z(i+1)= i$}
    \end{cases}
    &&\quad\text{for $z \in \bIfpf$.}
    \ea
\]
An \emph{involution Hecke word} 
for $z \in \bI$
is a word $a_1a_2\cdots a_p$
such that \[ m_1 \pi_{a_1}\pi_{a_2}\cdots \pi_{a_p} = \beta^N m_z \in \IdInvol\quad\text{for some integer $N\geq 0$.}\]
To avoid excessive subscripts, define 
\[ \mFPF := m_{\bFPF} \in \IdFixed.\]
An \emph{involution Hecke word}
for $z \in \bIfpf$ is a word $a_1a_2\cdots a_p$ 
such that  \[ \mFPF  \pi_{a_1}\pi_{a_2}\cdots \pi_{a_p} = \beta^N m_z \in \IdFixed\quad\text{for some integer $N\geq 0$,}\]
assuming $\beta^N \neq 0$ for $N\geq 0$.
Neither of these definitions depends on $\beta$, but in the fixed-point-free case we wish to exclude
words $a_1a_2\cdots a_p$ for which $z := s_{a_{i-1}}\cdots s_{a_2}s_{a_1} \bFPF s_{a_1}s_{a_2}\dots s_{a_{i-1}}$ has $a_i  +1 = z(a_i) >z(a_i+1) = a_i$ for some $i$.

Let $\InvolHecke(z)$ denote the set of involution Hecke words for 
 an element $z$ in $ \bI$ or $ \bIfpf$.
 This set was denoted as either
 $\iH(z)$ for $z \in \bI$ or $\cHfpf(z)$ for $z \in \bIfpf$ in \cite{Marberg2019a}.
Also define 
\[\ellhat(z) = \min\{ \ell(a) : a \in \InvolHecke(z)\}.\]
For an explicit formula for $\ellhat$, see \cite[Eq.\ (5.1)]{Marberg2019a}.

\begin{example} \label{ex:inv-hecke}
If $y =  s_3s_2s_3=s_2s_3s_2 =(2,4)\in \bI$, then
  $\InvolHecke(y)$ is the set of all finite words on the alphabet $\{2,3\}$ in which $2$ and $3$ both appear. If $w = (2,3,4)=s_2 s_3 \in S_n$ and $z = w^{-1} \bFPF w = \cdots (-3,-2)(-1,0)(1,4)(2,3)(5,6)(7,8) \cdots \in \bIfpf$, then $\InvolHecke(z)$ is the set of words obtained by prepending $2$ to a nonempty word on $\{1,3\}$. In either case $\ellhat(y)=\ellhat(z)=2$.
\end{example}

Our final theorem concerns these analogues of $\fkG_w$:

\begin{definition}
The \emph{involution Grothendieck polynomial} of $z \in \bI\sqcup \bIfpf$ is
\[\biG_z := \sum_{a \in \InvolHecke(z)} \sum_{0<\bfi \in \CS(a)} \beta^{\ell(\bfi)-\ellhat(z)} x_\bfi \in \ZZ[\beta][x_1,x_2,\dots,x_{n-1}].\]
\end{definition}

If $n$ is even and $z \in \bIfpf$ then $\biG_z$ coincides with the \emph{symplectic Grothendieck polynomials}
$\Gfpf_z$ studied in \cite{MP2019a,MP2019b}.
The paper \cite{MP2019a} also introduces certain \emph{orthogonal Grothendieck polynomials} $\iG_z$ indexed by 
$z \in \bI$, but these are generally not the same as $\biG_z$.
However, $\biG_z$ does specialize when $\beta=0$ to both kinds of \emph{involution Schubert polynomials} $\iS_z$ and $\Sfpf_z$
considered in \cite{HMP1,HMP6}.

Because $\IdInvol$ and $\IdFixed$ are $\cH$-modules,
there exists for each $z \in \bI\sqcup\bIfpf$ 
 a set $\A(z) \subset S_n$ (see \cite[\S2.1]{Marberg2019a}) such that 
\be\label{big-atoms-eq} \InvolHecke(z) = \bigsqcup_{w \in \A(z)} \Hecke(w)
\quand
\biG_z = \sum_{w \in \A(z)} \beta^{\ell(w) -\ellhat(z)} \fkG_w
\ee
where   
$\fkG_w := \bG_w(\dots,0,0,x_1,x_2,\dots,x_{n-1}) =  \sum_{a \in \Hecke(w)} \sum_{0<\bfi \in \CS(a)} \beta^{\ell(\bfi)-\ell(w)} x_\bfi$
for $w \in S_n$.

Again let $\bh_i(x) := 1+ x\pi_i \in \cH$ and define
\be
\bA_i(x) := \bh_{n-1}(x)\bh_{n-2}(x)\cdots \bh_i(x) \quand
{\widetilde A}^{(\beta)}_i(x) := \bh_i(x) \bh_{i+1}(x)\cdots \bh_{n-1}(x)
\ee
for integers $1\leq i < n$ and $x \in R$. Then consider the finite product
\be \fkG := \bA_1(x_1)\bA_2(x_2)\cdots \bA_{n-1}(x_{n-1}) = \sum_{w \in S_n} \fkG_w \cdot \pi_w \in \cH. \ee
Next let $\biG := m_1 \fkG $ and  $\biGfpf := \mFPF \fkG $.
It is evident from \eqref{big-atoms-eq} that
\[ \biG= \sum_{z \in \bI} \biG_z \cdot m_z \in \IdInvol
\quand
\biGfpf   = \sum_{z \in \bIfpf} \biG_z \cdot m_z \in \IdFixed.
\]
Proposition~\ref{prop1} is inefficient for computing $\biG_z$ since 
while $\fkG$ contains ${n \choose 2}$ factors $\bh_i(x_i)$, it turns out that any $m_z$ 
 can be written in the form $m_z\pi_{a_1}\pi_{a_2} \cdots \pi_{a_p}$ where $ p \leq \binom{n_1}{2} +\binom{n_2}{2}$
for $n_1= \lceil \frac{n+1}{2} \rceil$ and $n_2= \lfloor \frac{n+1}{2} \rfloor $. 
We can derive an involution version of Proposition~\ref{prop1}, however.

%\begin{lemma}[{\cite[Lem. 3]{KirillovNaruse}}]
%\label{groth-commute-lem}
%The elements $\bA_i(x)$ and $ {\widetilde A}^{(\beta)}_i(y)$ commute in $\cH$.
%\end{lemma}

\begin{lemma}
\label{groth-commute-lem2}
For any integer $1\leq i < n$ and elements $x_i,\dots,x_{n-1},y \in R$ it holds that 
\[{\widetilde A}^{(\beta)}_i(y) \bA_i(x_{i})\bA_{i+1}(x_{i+1})\cdots \bA_{n-1}(x_{n-1}) = \prod_{j=i+1}^{n-1} \bA_{j}(x_{j-1}) \cdot \prod_{j=i}^{n-1} \bh_j(x_j\oplus y). \]
\end{lemma}

\begin{proof}
Repeat the proof of \cite[Lem. 4.1]{FominStanley} with the symbols $A_i$, $\widetilde A_j$, $h_k$ replaced by
$\bA_i$, ${\widetilde A}^{(\beta)}_j$, $\bh_k$, and then apply the algebra anti-automorphism of $\cH$ that maps $\pi_w \mapsto \pi_{w^{-1}}$ to both sides.
\end{proof}

For integers $i>j>0$, define
$ x_{i\oplus j} = x_{j\oplus i} := x_i\oplus x_j =x_i + x_j + \beta x_ix_j$ and $  x_{j\oplus j} := x_j.$

 \begin{proposition} 
\label{prop:iSformula} 
The following identities hold:
\ben
\item[(i)] We have 
$\biG = \prod_{i=1}^{n-1}\prod_{j=\min(i,n-i)}^1 \bh_{i+j-1}(x_{i\oplus j}).$
\item[(ii)] If $n$ is even then
$\biGfpf = \prod_{i=2}^{n-1}\prod_{j=\min(i-1,n-i)}^1 \bh_{i+j-1}(x_{i \oplus j}).$
\een
\end{proposition}

In part (ii), the indices $i$ and $j$ in the products always satisfy $i>j >0$ so $x_{i\oplus j} = x_i\oplus x_j$.

\begin{proof}
   We first prove part (i). The result is trivial when $n=1$ so assume $n\geq 2$. For any $1 \leq i <n$ we have 
   $
   m_1\pi_i \pi_{i+1} = m_{s_{i+1}s_i s_{i+1}} = m_1 \pi_{i+1}\pi_{i}
   $ and consequently
$
   m_1 \bh_i(x)\bh_{j}(y) = m_1\bh_{j}(y)\bh_i(x) 
  $ for all integers $i$, $j$ and $x,y \in R$. Using this, one checks that
$
   m_1 \bA_1(x) = m_1 {\widetilde A}^{(\beta)}_1(x)
$, whence 
 \[\ba
        \biG &= m_1 \bA_1(x_1)\bA_2(x_2)\cdots \bA_{n-1}(x_{n-1}) 
        = m_1{\widetilde A}^{(\beta)}_1(x_1)\bA_2(x_2)\cdots \bA_{n-1}(x_{n-1})\\
        & = m_1 \bh_1(x_1) {\widetilde A}^{(\beta)}_2(x_1)\bA_2(x_2) \cdots \bA_{n-1}(x_{n-1}).
%        & = m_1h_1(x_1)A_2(x_2)\tilde{A}_2(x_1) \cdots A_{n-1}(x_{n-1}) \quad \text{(by Lemma ...)}\\
%        & = m_1h_1(x_1)A_2(x_2)h_2(x_1)\tilde{A}_3(x_1)A_3(x_3) \cdots A_{n-1}(x_{n-1})\\
%        & = m_1h_1(x_1)A_2(x_2)h_2(x_1)A_3(x_3)h_3(x_1) \cdots A_{n-1}(x_{n-1})h_{n-1}(x_1)\\
%        &= m_1h_1(x_1)A_3(x_2)h_2(x_2+x_1)A_4(x_3)h_3(x_3+x_1) \cdots A_{n-1}(x_{n-2})h_{n-1}(x_{n-1}+x_1)\\   
%        &= m_1A_3(x_2)A_4(x_3)\cdots A_{n-1}(x_{n-2}) \cdot h_1(x_1)h_2(x_2+x_1) \cdots h_{n-1}(x_{n-1}+x_1) 
\ea
\]
Applying Lemma~\ref{groth-commute-lem2} with $i=2$ and commuting $\bh_1(x_1)$ all the way to the right 
gives
\[ \biG =m_1\bA_3(x_2)\bA_4(x_3)\cdots \bA_{n-1}(x_{n-2})  \bh_1(x_{1\oplus 1})\bh_2(x_{1\oplus 2}) \cdots \bh_{n-1}(x_{1\oplus (n-1)}) .\]
We may assume by induction that 
\[\ba m_1\bA_3(x_2)\bA_4(x_3) \cdots \bA_{n-1}(x_{n-2}) 
&= m_1\prod_{i=1}^{n-3} \prod_{j=\min(i,n-2-i)}^1 \bh_{i+j+1}(x_{(i+1)\oplus (j+1)})
\\&= m_1\prod_{i=2}^{n-2} \prod_{j=\min(i,n-i)}^2 \bh_{i+j-1}(x_{i\oplus j}).
\ea\]
This gives
$
    \biG =    m_1\prod_{i=2}^{n-2} \prod_{j=\min(i,n-i)}^2 \bh_{i+j-1}(x_{i\oplus j}) \cdot \prod_{k=1}^{n-1} \bh_k(x_{1 \oplus i})
$,
and
    it is not hard to see that this formula can be transformed by appropriate commutations to 
    the expression in part (i).
%    $
%     \biG =        m_1 h_1(x_{1})\prod_{i=2}^{n-2} \prod_{j=\min(i,n-i)}^1 h_{i+j-1}(x_{i\oplus j}) \cdot h_{n-1}(x_{i\oplus (n-1)})
%$
%    which is the desired expression. 
    For instance, if $n = 8$ then what needs to be shown is 
    equivalent to the claim that one can turn the reduced word
  $
        3\cdot 54 \cdot 765 \cdot 76 \cdot 7\cdot 1234567
$ into 
$
 1 \cdot 32 \cdot 543 \cdot 7654 \cdot 765 \cdot 76 \cdot 7$
using only relations of the form $ij \leftrightarrow ji$ for $|i-j|>1$.

The proof of part (ii) is similar.
Assume $n$ is even and $1 \leq i <n$. If $i$ is odd then 
%   \be
%   \mFPF \pi_i = 0\text{ if $i$ is odd}\quand 
%   \mFPF\pi_i \pi_{i-1}=  \mFPF\pi_i \pi_{i-1}\text{ if $i$ is even}.
%   \ee Thus 
$ \mFPF \pi_i = 0 $ and $\mFPF \bh_i(x) = \mFPF$ for all $x \in R$.
On the other hand, if $i$ is even and $x,y\in R$ then
\[   
    \mFPF\pi_i \pi_{i+1}=  \mFPF\pi_i \pi_{i-1}
    \quand
   \mFPF \bh_i(x)\bh_{i+1}(y) = \mFPF \bh_i(x)\bh_{i-1}(y).
\]    
 Using these relations repeatedly we deduce that
$
   \mFPF \bA_i(x) = \mFPF {\widetilde A}^{(\beta)}_{i+1}(x)
$
for any odd integer $1\leq i < n$.
   By Lemma~\ref{groth-commute-lem2}, we therefore have
 \[\ba
        \biGfpf &= \mFPF A_1(x_1)A_2(x_2)\cdots A_{n-1}(x_{n-1}) =
      %  &= \mFPF\widetilde{A}_1(x_1)A_2(x_2)\cdots A_{n-1}(x_{n-1})\\
         \mFPF {\widetilde A}^{(\beta)}_2(x_1)\bA_2(x_2) \cdots \bA_{n-1}(x_{n-1}) \\
         &=
         \mFPF \bA_3(x_2)\bA_4(x_3)\cdots \bA_{n-1}(x_{n-2}) \cdot \bh_2(x_1\oplus x_2) \bh_3(x_1\oplus x_3) \cdots \bh_{n-1}(x_1\oplus x_{n-1}).
\ea
\]
From here, the result follows by induction as in the proof of part (i).
\end{proof}

Let $\ltriang_n := \{(i,j) \in \ZZ \times \ZZ : i \geq j>0\}$ and 
$ \ltriangneq_n :=  \{(i,j) \in \ZZ \times \ZZ : i > j>0\}$.
Equip these sets with the total order defined by $(i,j) \prec (k,l)$ if $i<k$ or if $i=k$ and $j > l$. 
An \emph{involution Hecke pipe dream} for $z \in \bI$
(respectively, $z \in \bIfpf$)
 is a finite subset $D$ of $\ltriang_n$ (respectively, $\ltriangneq_n$)
 such that the word formed by listing the numbers $i+j-1$ as $(i,j)$ runs over $D$ in the order $\prec$
 belongs to $\InvolHecke(z)$.
We write $\IP(z)$ for the set of these subsets.

\begin{theorem} \label{thm:inv-pipe-dream-formula} 
If $z \in \bI$ or if $n$ is even and $z \in \bIfpf$ then
\[
 \biG_z = \sum_{D \in \IP(z)} \beta^{|D| -\ellhat(z)} \prod_{(i,j) \in D} x_{i\oplus j}
\]
where we set $x_{i\oplus i} := x_i$ for $i>0$ and $x_{i\oplus j}  := x_i + x_j + \beta x_ix_j$ for $i>j>0$.
\end{theorem}

%Recall  that  $x_{i\oplus i} := x_i$ for $i>0$ and $x_{i\oplus j}  := x_i + x_j + \beta x_ix_j$ for $i>j>0$.
When $\beta=0$ our result reduces to \cite[Thm. 1.5]{HMP6}, 
which was proved in a different way using somewhat involved recurrences.
The methods here give a new and arguably simpler proof.
For generic $\beta$, Theorem~\ref{thm:inv-pipe-dream-formula}
resolves the symplectic half of \cite[Problem 6.9]{HMP6}.

\begin{proof} First assume $z \in \bI$. Part (i) of Proposition~\ref{prop:iSformula} implies
\[
        \biG_z = \sum_{a=a_1\cdots a_N \in \InvolHecke(z)} \beta^{N -\ellhat(z)}\sum_{
       \substack{ 0 < \bfi=(i_1 \leq \dots \leq i_N) \in \CS(a) \\
         i_j \leq a_j < 2i_j\ \forall j}} x_{i_1\oplus(a_1-i_1+1)}\cdots x_{i_N\oplus(a_N-i_N+1)}.
\]
    One now checks that the map sending $(a,\bfi) $ to $D=\{(i_j,a_j-i_j+1) : 1 \leq j \leq \ell(a)\}$ is a bijection from the pairs indexing this double summation to 
    the elements of $\IP(z)$.
When $n$ is even and $z \in \bIfpf$, the same argument using part (ii) of Proposition~\ref{prop:iSformula}
gives the desired formula.
\end{proof}

\begin{example}
Suppose $n=4$. If $y = s_3s_2s_3 =s_2s_3s_2 = (2,4) \in \bI$ as in
  Example~\ref{ex:inv-hecke}, then the elements of $\IP(y)$ are the sets of nonzero positions in the matrices
  \begin{equation*} \small
  \begin{bmatrix} 0 & 0 & 0 \\ 1 & 1 & 0 \\ 0 & 0 & 0 \end{bmatrix}, \quad 
  \begin{bmatrix} 0 & 0 & 0 \\ 1 & 0 & 0 \\ 1 & 0 & 0 \end{bmatrix}, \quad 
  \begin{bmatrix} 0 & 0 & 0 \\ 1 & 1 & 0 \\ 1 & 0 & 0 \end{bmatrix},
  \end{equation*}
which are explicitly $\{(2,1),(2,2)\}$, $\{(2,1),(3,1)\}$, and $\{(2,1),(2,2),(3,1)\}$. By Theorem~\ref{thm:inv-pipe-dream-formula},
  \begin{equation*}
    \biG_{y} = (x_2 \oplus x_1)x_2 + (x_2 \oplus x_1)(x_3 \oplus x_1) + \beta(x_2 \oplus x_1)x_2 (x_3 \oplus x_1).
  \end{equation*}
Alternatively, if $z = s_3\cdot s_2\cdot \bFPF \cdot s_2\cdot s_3 \in \bIfpf$ as in Example~\ref{ex:inv-hecke}, 
  then $\IP(z)$ contains just one element $\{(2,1),(3,1)\}$, and
  Theorem~\ref{thm:inv-pipe-dream-formula} asserts
that
  $\biG_z = (x_2 \oplus x_1)(x_3 \oplus x_1)$.
\end{example}


\begin{thebibliography}{99}

\bibitem{BH} S. C. Billey and M. Haiman, Schubert polynomials for the classical groups,
\emph{J. Amer. Math. Soc.} \textbf{8} (1995), pp. 443--482.

\bibitem{BHY} S. C. Billey, A. E. Holroyd, and B. Young, A bijective proof of Macdonald's reduced word formula, 
\emph{Algebraic Combinatorics} \textbf{2} (2019) no. 2, pp. 217--248

\bibitem{BJS} S. C. Billey, W. Jockusch, and R. P. Stanley, Some Combinatorial Properties of Schubert Polynomials, 
\emph{J. Algebr. Combin.} \textbf{2} (1993), pp. 345--374.

\bibitem{Buch} A. S. Buch, A Littlewood-Richardson rule for the $K$-theory of Grassmannians, \emph{Acta Math.} \textbf{189} (2002), no. 1, 37--78.


\bibitem{FominKirillov} S. Fomin and A. N. Kirillov, Reduced Words and Plane Partitions,
\emph{J. Algebr. Combin.} \textbf{6} (1997), pp. 311-319.

\bibitem{FominStanley} S. Fomin and R. P. Stanley, Schubert Polynomials and the NilCoxeter Algebra,
\emph{Adv. Math.} \textbf{103} (1994), pp. 196--207.

\bibitem{HMP1} Z. Hamaker, E. Marberg, and B. Pawlowski, Involution words: counting problems and connections to Schubert calculus for symmetric orbit closures, \emph{J. Combin. Theory Ser. A} \textbf{160} (2018), 217--260.

\bibitem{HMP6} Z. Hamaker, E. Marberg, and B. Pawlowski, Involution pipe dreams, preprint (2019), {\tt arXiv: 1911.12009}.

\bibitem{KirillovNaruse} A. N. Kirillov and H. Naruse, Construction of Double Grothendieck Polynomials of Classical Types using IdCoxeter Algebras, 
\emph{Tokyo J. Math.} 39.3 (2017), 695--728.

\bibitem{LamLeeShim} T. Lam, S. Lee, and M. Shimozono, Back stable Schubert calculus, \emph{Compositio Math.} \textbf{157} (2021), 883--962.
%preprint (2018), {\tt arXiv:1806.11233}.

\bibitem{Macdonald} I. G. Macdonald, \emph{Notes on Schubert Polynomials}, 
Laboratoire de combinatoire et d'informatique math\'ematique (LACIM), Universit\'e du Qu\'ebec \`a Montr\'eal, Montreal, 1991.

%\bibitem{Macdonald2} I. G. Macdonald, Symmetric Functions and Hall Polynomials, 2nd ed., Oxford University Press, New York, 1999.

\bibitem{Marberg2019a} E. Marberg, A symplectic refinement of shifted Hecke insertion, 
\emph{J. Combin. Theory Ser. A} \textbf{173} (2020), 105216.
%preprint (2019), {\tt arXiv:1901.06771}.

\bibitem{MP2019a} E. Marberg and B. Pawlowski, $K$-theory formulas for orthogonal and symplectic orbit closures, 
\emph{Adv. Math.} \textbf{372} (2020), 107299.
%preprint (2019), {\tt arXiv:1906.00907}.

\bibitem{MP2019b} E. Marberg and B. Pawlowski, On some properties of symplectic Grothendieck polynomials, 
\emph{J. Pure Appl. Algebra} 225.1 (2021), 106463.
%preprint (2019), {\tt arXiv:1906.01286}.

\bibitem{Pawlowski2019} B. Pawlowski, Universal graph Schubert varieties, \emph{Transform. Groups} (2022).

\bibitem{Young} B. Young, A Markov growth process for Macdonald's distribution on reduced words,
preprint (2014), {\tt 	arXiv:1409.7714}.

\end{thebibliography}
\end{document}